\documentclass[10pt]{article}
\usepackage{amsthm, amsmath, amssymb, fullpage}
\pdfoutput=1
\newtheorem{thm}{Theorem}[section]
\newtheorem{cor}[thm]{Corollary}
\newtheorem{lem}[thm]{Lemma}
\newtheorem{prop}[thm]{Proposition}
\usepackage[pdftex]{graphicx,color}
\newcommand{\bZ}{\mathbb{Z}}

\begin{document}
\title{Dirichlet-Ford Domains and Arithmetic Reflection Groups}
\author{Grant S. Lakeland}
\date{}
\maketitle
\begin{abstract}
In this paper, it is shown that a Fuchsian group, acting on the upper half-plane model for $\mathbb{H}^2$, admits a Ford domain which is also a Dirichlet domain, for some center, if and only if it is an index $2$ subgroup of a reflection group. This is used to exhibit an example of a maximal arithmetic hyperbolic reflection group which is not congruence. Analogous results, and counterexamples, are given in the case of Kleinian groups. 
\end{abstract}
\section{Introduction}
The action of the modular group $\mathrm{(P)SL}_2(\bZ)$ on the upper half-plane model for hyperbolic $2$-space $\mathbb{H}^2$ has been extensively studied. It is well-known that a fundamental domain for this action is given by the triangle $T$ with vertices at $\rho=\frac{1}{2}+\frac{\sqrt{-3}}{2}$, $-\overline{\rho}$ and $\infty$. This domain is an example of two common constructions of fundamental domains for Fuchsian groups: it is both a Ford domain and a Dirichlet domain for the action of the modular group. Furthermore, it arises from more than one distinct choice of Dirichlet center, as taking the Dirichlet domain centered at any $z_0 = iy$, for $y>1$, gives rise to $T$. One expects the Dirichlet domain to change along with the choice of center \cite{Diaz}, so in this sense the modular group exhibits some atypical properties.\vspace{2 mm}\\
It is also well-known that $\mathrm{PSL}_2(\bZ)$ is the orientation-preserving index $2$ subgroup of the group generated by reflections in an ideal triangle in $\mathbb{H}^2$ with angles $\frac{\pi}{2}$, $\frac{\pi}{3}$ and $0$, located at $i$, $\rho$ and $\infty$ respectively. More generally, a \emph{hyperbolic reflection group} is a subgroup of $\mathrm{Isom}(\mathbb{H}^2)$ generated by reflections in the sides of a polygon $Q \subset \mathbb{H}^2$. Such a group is discrete if and only if each angle of $Q$ is either $0$ or an integer submultiple of $\pi$.\vspace{2mm}\\
The purpose of this paper is to determine exactly which Fuchsian groups admit a fundamental domain that is both a Dirichlet domain and a Ford domain (which we will call a \emph{DF domain}) or a Dirichlet domain for multiple centers (a \emph{Double Dirichlet domain}). It turns out that the above properties of $\mathrm{PSL}_2(\bZ)$ are very much related; we demonstrate the following result (Theorems \ref{5.3} and \ref{5.4}) regarding such groups:\vspace{2 mm}\\
\textbf{Theorem.} \textit{A finitely generated, finite coarea Fuchsian group $\Gamma$ admits a DF domain (or a Double Dirichlet domain) $P$ if and only if $\Gamma$ is an index $2$ subgroup of the discrete group $G$ of reflections in a hyperbolic polygon $Q$.}\vspace{2 mm}\\
The condition given by this result provides a method of checking whether a given Fuchsian group is the index $2$ orientation-preserving subgroup of a hyperbolic reflection group. This is particularly useful in the context of \emph{maximal arithmetic} reflection groups. A non-cocompact hyperbolic reflection group is \emph{arithmetic} if it is commensurable with $\mathrm{PSL}_2(\bZ)$, and \emph{maximal arithmetic} if it is not properly contained in another arithmetic reflection group. As an application of the above theorem, we give the following result:\vspace{2mm}\\
\textbf{Corollary.} \textit{There exists a maximal arithmetic hyperbolic reflection group which is not congruence.}\vspace{2mm}\\
This answers a question raised by Agol--Belolipetsky--Storm--Whyte in \cite{ABSW} (see also Belolipetsky \cite{Belo}).\vspace{2mm}\\
Any group $\Gamma$ satisfying the theorem must have \emph{genus zero} \cite{LMR}, as well as a certain symmetrical property. Having a Double Dirichlet or DF domain therefore also gives an obstruction to the group having \emph{non-trivial cuspidal cohomology} \cite{GrunSchw2}.\vspace{2 mm}\\
Motivated by this, one may ask whether there exists a similar obstruction to non-trivial cuspidal cohomology for Kleinian groups. We will show that this is not the case: in Section \ref{kleinian} we exhibit a Kleinian group which possesses both non-trivial cuspidal cohomology and a DF domain. However, the condition of having such a domain does impose some restrictions on $\Gamma$; perhaps the most striking is that the group possesses a generating set, all of whose elements have real trace (Theorem \ref{dim3thm}).\vspace{2 mm}\\
This paper is organized as follows. After the preliminaries of Section \ref{prelims}, Section \ref{df} will examine Fuchsian groups with DF domains, and show that such domains are symmetrical and give rise to punctured spheres. The more general case of the Double Dirichlet domain is discussed in Section \ref{dd}. In Section \ref{refgps}, it will be shown that the main theorem follows from the previous sections and standard results on reflection groups. An example of a non-congruence maximal arithmetic hyperbolic reflection group can be found in Section \ref{example}. Section \ref{kleinian} contains a discussion of these domains in the setting of Kleinian groups.\vspace{2 mm}\\
\textbf{Acknowledgments.} The author thanks his advisor, Alan Reid, for his guidance and encouragement, and Daniel Allcock and Hossein Namazi, for additional helpful conversation and correspondence.

\section{Preliminaries}\label{prelims}

We will work in the upper half-plane model for hyperbolic space. The group of conformal, orientation-preserving isometries (or linear fractional transformations) of $\mathbb{H}^2$ can be identified with $\mathrm{PSL}_2(\mathbb{R})$ via the correspondence
$$\begin{pmatrix} a & b \\ c & d\end{pmatrix}  \\ \ \ \  \longleftrightarrow \ \ \ z \longmapsto \frac{az+b}{cz+d}.$$
A \emph{Fuchsian group} $\Gamma$ is a subgroup of $\mathrm{PSL}_2(\mathbb{R})$, discrete with respect to the topology induced by regarding that group as a subset of $\mathbb{R}^4$. The \emph{Dirichlet domain} for $\Gamma$ centered at $z_0$ is defined to be 
$$\lbrace x \in \mathbb{H}^2 \ | \ d(x,z_0) \leq d(x, \alpha(z_0)) \ \forall \ 1 \neq \alpha \in \Gamma \rbrace.$$
It is an intersection of closed half-spaces.\vspace{2 mm}\\
Beardon (\cite{Beardon}, Section 9.5) demonstrates an alternative definition, in terms of reflections, which allows us to define a \emph{generalized Dirichlet domain} by taking our center to be on the boundary $\partial \mathbb{H}^2$. We will typically conjugate $\Gamma$ in $\mathrm{PSL}_2(\mathbb{R})$ so that this center is placed at $\infty$ in the upper half-plane. We suppose $\Gamma$ is \emph{zonal}, or that $\infty$ is a parabolic fixed point, and so the reflections given are not uniquely determined for any parabolic isometry fixing $\infty$. To account for this, we define a \emph{Ford domain} \cite{Ford} to be the intersection of the region exterior to all isometric circles with a fundamental domain for the action of the parabolic subgroup stabilizing $\infty$, $\Gamma_\infty < \Gamma$, which is a vertical strip.\vspace{2 mm}\\
For a given finitely generated Fuchsian group $\Gamma$, the \emph{signature} $(g; n_1, \ldots, n_t; m; f)$ of $\Gamma$ records the topology of the quotient space $\mathbb{H}^2 / \Gamma$, where $g$ is the genus, $t$ is the number of cone points of orders $n_1, \ldots, n_t$ respectively, $m$ is the number of cusps, and $f$ is the number of infinite area funnels. If $\Gamma$ is the orientation-preserving index $2$ subgroup of a reflection group, then $\mathbb{H}^2 / \Gamma$ is a sphere with cusps and/or cone points, and thus in this case we have $g=0$. If additionally $\Gamma$ has finite coarea, then we also have that $f=0$.\vspace{2mm}\\
The group of orientation-preserving isometries of the upper half-space model of $\mathbb{H}^3$ can likewise be identified with $\mathrm{PSL}_2(\mathbb{C})$. A \emph{Kleinian group} is a discrete subgroup of this isometry group. The definitions of Dirichlet domain and Ford domain carry over to this situation, with one small modification: instead of $\Gamma$ being zonal, we assume that $\Gamma_\infty$ contains a copy of $\mathbb{Z}^2$.\vspace{2 mm}\\
Throughout, we will assume that $\Gamma$ is finitely generated, and hence that all fundamental domains we encounter have a finite number of sides. For simplicity, we will also suppose that $f=0$ and that $\Gamma$ has finite covolume (and thus that all fundamental domains have finite volume; that is, finitely many ideal vertices, each adjacent to two sides), although many of the arguments should extend to the case where $\Gamma$ does not have finite covolume.

\section{DF Domains}\label{df}

Suppose $\Gamma$ contains a non-trivial parabolic subgroup $\Gamma_\infty$ fixing $\infty$. In $\mathbb{H}^2$, $\Gamma_\infty$ must be cyclic, and after conjugation, we may take it to be generated by
$$T = \begin{pmatrix} 1 & 1 \\ 0 & 1 \end{pmatrix}.$$

\begin{thm}\label{PS} If $\Gamma$ admits a DF domain, then the quotient space $\mathbb{H}^2 / \Gamma$ is a punctured sphere, possibly with cone points.\end{thm}

Before commencing the proof of this, we will prove two elementary but important lemmas.

\begin{lem}\label{Lemma1} \emph{(see \cite{Beardon}, Section 9.6.)} Any vertex cycle on the boundary of a Ford domain $P$ is contained within a horocycle based at $\infty$.\end{lem}

\begin{proof} Fix a vertex $v$. By construction of $P$, $v$ lies on or exterior to all isometric circles, and necessarily lies on at least one. We first consider a $\gamma \in \Gamma$ such that $v \notin S_\gamma$. Then $v$ lies exterior to $S_\gamma$. It follows that $\gamma$ sends $v$ into the interior of $S_{\gamma^{-1}}$. Thus $\gamma(v)$ cannot be a vertex of $P$. Now suppose that $v \in S_\gamma$. Then $\gamma$ is the composition of reflection in $S_\gamma$, which fixes $v$, and reflection in a vertical line. It therefore necessarily preserves the imaginary part of $v$, proving the lemma. \end{proof}

{\bf Remark.} The lemma holds for any point on the boundary of the Ford domain $P$. For our purposes, it will be enough to have it for the vertices of $P$.

\begin{lem}\label{Lemma2} \emph{(see \cite{Greenberg}, p. 203.)} Let $P$ be a Dirichlet domain for $\Gamma$ with center $z_0$. Let $1 \neq \gamma \in \Gamma$ and suppose that $z$, $\gamma(z) \in \partial P \cap \mathbb{H}^2$. Then $d_\mathbb{H}(z,z_0) = d_\mathbb{H}(\gamma(z),z_0)$.\end{lem}

\begin{proof} This is an application of the definition of a Dirichlet domain stated above. Specifically, setting $x = z$ and $\alpha = \gamma^{-1}$ yields the inequality
$$d(z,z_0) \leq d(z,\gamma^{-1}(z_0)) = d(\gamma(z),z_0),$$
the latter equality holding because $\gamma$ is an isometry. Setting $x = \gamma(z)$ and $\alpha = \gamma$ now gives
$$d(\gamma(z),z_0) \leq d(\gamma(z),\gamma(z_0)) = d(z,z_0).$$
Combining these two inequalities gives the required equality. \end{proof}

We will now use these two lemmas to prove Theorem 3.1.

\begin{proof} Suppose we are given a DF domain $P$ for $\Gamma$. Since $P$ is a Ford domain, it is contained in a fundamental region for $\Gamma_\infty$, which is a vertical strip 
$$\lbrace z \in \mathbb{H}^2 \ | \ x_0 \leq \text{Re}(z) \leq x_0 + 1 \rbrace$$
for some $x_0 \in \mathbb{R}$. Shimizu's Lemma (see \cite{Maskit}, p. 18) tells us that the radii of the isometric circles $S_\gamma$ cannot exceed $1$. Thus we may consider a point $z = x_0 + iy \in \partial P$, where $y>1$. Choosing $\gamma=T$, and applying Lemma \ref{Lemma2} to $z$ and $\gamma(z)$, we find that $\text{Re}(z_0) = x_0 + \frac{1}{2}$.\vspace{2 mm}\\
Next suppose that $v \in \mathbb{H}^2$ is a vertex of $P$, and $\gamma \in \Gamma$ a side-pairing such that $\gamma(v)$ is another vertex of $P$. Then, by Lemma \ref{Lemma1}, $\text{Im}(\gamma(v)) = \text{Im}(v)$, and by Lemma \ref{Lemma2}, $d_\mathbb{H}(\gamma(v),z_0) = d_\mathbb{H}(v,z_0)$. Consider the two sets $\{z \in \mathbb{H}^2 \ | \ \text{Im}(z) = \text{Im}(v) \}$ and $\{z \in \mathbb{H}^2 \ | \ d_\mathbb{H}(z,z_0) = d_\mathbb{H}(v,z_0) \}$. The former is the horizontal line through $v$, and the latter a circle with Euclidean center located vertically above $z_0$ (see Figure \ref{figure:figure03}). In particular, the picture is symmetrical in the vertical line $\{ \text{Re}(z) = x_0 + \frac{1}{2} \}$. Either $\gamma(v) = v$ or $\gamma(v) = v^\ast$, where $v^\ast$ is the reflection of $v$ in the line $ \{ \text{Re}(z) = x_0 + \frac{1}{2} \}$.\vspace{2 mm}\\
Suppose that $\gamma(v) = v$. Then, by considering a point $w \in \partial P$ close to $v$, the fact that $d(w,z_0) = d(\gamma(w),z_0)$ means that $v$ necessarily lies directly below the Dirichlet center $z_0$. The contrapositive of this states that if $\text{Re}(v) \neq x_0 + \frac{1}{2}$, then any side-pairing $\gamma$ pairing $v$ with a vertex of $P$ must send $v$ to $v^\ast$.
\begin{figure}[htbp]
\begin{center}
 
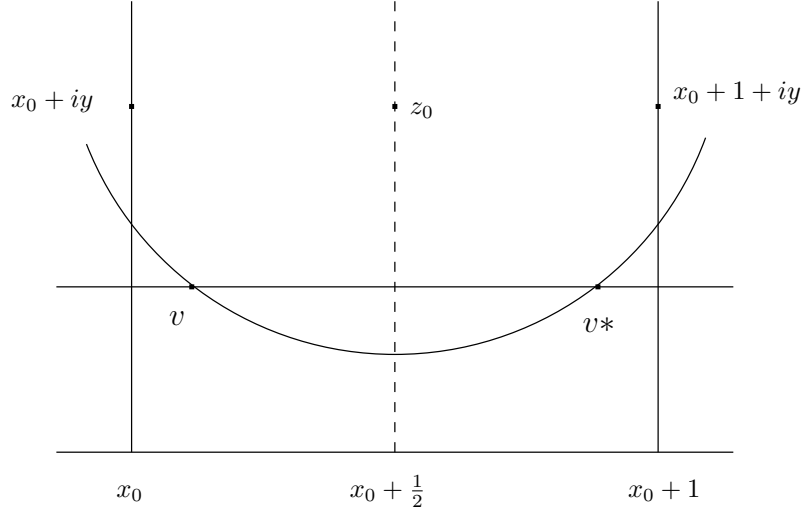
 
\caption{$\gamma(v)=v^*$}
\label{figure:figure03}
\end{center}
\end{figure}
\vspace{2 mm}\\
Suppose now that $v \in \partial \mathbb{H}^2$ is a vertex of P. Then two isometric circles meet at $v$. Fix one such circle $S$. $S$ is the isometric circle $S_\gamma$ of some element $\gamma \in \Gamma$. $S_\gamma$ contains a side of $P$ adjacent to $v$, and we pick two points of $S_\gamma$, $w_1$, $w_2 \in \partial P \cap \mathbb{H}^2$. By Lemma \ref{Lemma2}, $\gamma$ must send both $w_1$ and $w_2$ to points the same respective distances from $z_0$. For each $i$, $w_i$ is either fixed or sent to its reflection in the line $\{ \text{Re}(z) = \text{Re}(z_0) \}$. If $w_1$ were fixed, $w_2$ would neither be fixed nor sent to its reflection, and vice-versa if $w_2$ were fixed. Thus we conclude that $\gamma$ sends points of $S$ to their reflections in the line $\{ \text{Re}(z) = \text{Re}(z_0) \}$.\vspace{2 mm}\\
We can now show that $\mathbb{H}^2 / \Gamma$ is a punctured sphere. We first identify the two vertical sides of $P$, creating the cusp at $\infty$ and a circle awaiting identification. Consider the side of $P$ adjacent to the side contained in the vertical line $\text{Re}(z) = x_0$. This side lies on some isometric circle $S_\gamma$. We see that $\gamma$ must identify our side with a side adjacent to the side of $P$ contained in the line $\text{Re}(z) = x_0 + 1$. Working inwards toward the center and applying this argument repeatedly, we see that all sides must pair up symmetrically. In particular, there can not exist two hyperbolic generators whose axes intersect precisely once. Thus we conclude that the quotient space has genus zero.\end{proof}

{\bf Remarks. (1)} We may take the Dirichlet center of $P$ to be any point of the interior of $P$ on this vertical line $\lbrace \mathrm{Re}(z)=x_0+\frac{1}{2} \rbrace$. To see this, let $z_0$ be any such point, and $\gamma \in \Gamma \ \backslash \ \Gamma_\infty$ a side-pairing of $P$. Since $\gamma(S_\gamma) = S_{\gamma^{-1}}$, and this pair are arranged symmetrically with respect to the line $\lbrace \mathrm{Re}(z)=x_0+\frac{1}{2} \rbrace$, both of these isometric circles are geodesics of the form used to construct the Dirichlet polygon centered at $z_0$.

{\bf (2)} The converse of Theorem \ref{PS} is false. The symmetrical nature of $P$ implies a certain symmetry in the quotient space $\mathbb{H}^2 / \Gamma$, namely that the surface admits an orientation-reversing involution of order 2. This is not something we see in a generic punctured sphere.

\section{Double Dirichlet Domains}\label{dd}

We now suppose that the same fundamental domain $P$ is obtained when we construct the Dirichlet domains $P_0$ and $P_1$ centered at $z_0$ and $z_1 \in \mathbb{H}^2$ respectively. For comparison with the previous section, we will assume that we have conjugated $\Gamma$ in $\text{PSL}_2(\mathbb{R})$ so that the geodesic line $L$ containing $z_0$ and $z_1$ is vertical.

\begin{thm}\label{DD} If the Dirichlet domains $P_0$ and $P_1$ for $\Gamma$, centered at $z_0 \neq z_1 \in \mathbb{H}^2$ respectively, coincide, then the quotient space $\mathbb{H}^2 / \Gamma$ is a sphere, with cone points and/or punctures.\end{thm}

\begin{proof} Much of the work in Section $3$ was concerned with showing precisely how the sides of $P$ were identified. This follows relatively swiftly here, once we have cleared up one technical point. We often think of a fundamental domain as a subset of $\mathbb{H}^2$ combined with a set of side-pairings identifying its sides. We only assume that the sets $P_0$ and $P_1$ are equal, and thus we must make sure that $\Gamma$ identifies their sides the same way each time.

\begin{lem}\label{Match} If $P = P_0 = P_1$ is the Dirichlet domain centered at $z_0$ and at $z_1$, then the sides of $P$ are identified the same way in each case.\end{lem}

\begin{proof} Suppose, for the sake of contradiction, that this is not the case. Any side of a Dirichlet domain bisects the domain's center and its image under some isometry. Here, we have a side $A$ of $P$ which is the bisector of both the pair $z_0$ and $\gamma_0^{-1}(z_0)$ and the pair $z_1$ and $\gamma_1^{-1}(z_1)$, where $\gamma_0 \neq \gamma_1$ are the isometries defining that side of $P$. It follows that $\gamma_0$ pairs $A$ with some side $B$, and $\gamma_1$ pairs $A$ with some other side $C \neq B$. Let $d := d(z_0,z_1)$ be the distance between the two centers $z_0$ and $z_1$. Since $\gamma_0^{-1}(z_0)$ and $\gamma_1^{-1}(z_1)$ are the reflections of each in $A$, we see that
$$d(\gamma_0^{-1}(z_0),\gamma_1^{-1}(z_1)) = d.$$
Applying the isometry $\gamma_1$ to both points, this gives that
$$d(\gamma_1(\gamma_0^{-1}(z_0)),z_1) = d.$$
Now, if $\gamma_1(\gamma_0^{-1}(z_0)) = z_0$, then the isometries $\gamma_0$ and $\gamma_1$ both send $\gamma_0^{-1}(z_0)$ to $z_0$ and $\gamma_1^{-1}(z_1)$ to $z_1$. Since they also both preserve orientation, this implies that $\gamma_0 = \gamma_1$, a contradiction. Thus $\gamma_1(\gamma_0^{-1}(z_0)) \neq z_0$. But then $\gamma_1(\gamma_0^{-1}(z_0))$ is a point in the orbit of $z_0$, and thus the construction of $P_0$ involves the half-space $\{ x \in \mathbb{H}^2 \ | \ d(x,z_0) \leq d(x,\gamma_1(\gamma_0^{-1}(z_0))) \}$. As we saw above,
$$d(\gamma_1(\gamma_0^{-1}(z_0)),z_1) = d(z_0,z_1) = d.$$
Hence $z_1$ is equidistant from $z_0$ and $\gamma_1(\gamma_0^{-1}(z_0))$. Thus $z_1$ cannot be in the interior of $P_0$, contradicting the assumption that $P_0 = P_1$.\end{proof}

The following result will allow us to conclude the proof of Theorem \ref{DD}.

\begin{lem}\label{Glue} Each side of $P$ (and each point of $\partial P$) is identified with its reflection in the line $L$.\end{lem}

\begin{proof} Given a point $v \in \partial P$, $v$ is sent to a point of $\partial P$ the same distance away from $z_0$. Put another way, $v$ is sent somewhere on the hyperbolic circle of center $z_0$ and radius $d(v,z_0)$. $v$ is also sent to a point on the hyperbolic circle of center $z_1$ and radius $d(v,z_1)$. Thus we see a picture similar to Figure \ref{figure:figure03}, except instead of a horizontal line, we have a second circle, centered vertically above or below $z_0$. These two circles intersect only at $v$ if $v \in L$, and at $v$ and $v^*$, the reflection of $v$ in $L$, if $v \notin L$. If $v \in L$ then $v$ is necessarily an elliptic fixed point and a vertex of $P$, and the two sides adjacent to $v$ are identified with one another. If $v \notin L$, it suffices to show that $v$ cannot be fixed by a side-pairing, and thus must be identified with $v^*$. If $v$ is a vertex of $P$ this follows by considering a sequence of points on a fixed side $A$ adjacent to $v$.\end{proof}

So we now know that our domain $P$ has the same symmetrical property that we saw DF domains possess. If the line $L \cap \textit{\r{P}}$ extends vertically to $\infty$, then the argument from the proof of Theorem \ref{PS} applies directly, and we are done. If the line terminates at a boundary point of $P$, then we observe that the two sides adjacent to this vertex are identified symmetrically, creating a cone point instead of a cusp. This creates a boundary circle as in the proof of Theorem \ref{PS}, and the rest of the argument applies from there.\end{proof}

\textbf{Remarks. (1)} The first remark at the end of Section $3$ applies here as well. That is, if we take any point $z \in L \cap \textit{\r{P}}$ as our Dirichlet center, we will obtain the Dirichlet domain $P$. Thus we see that a Fuchsian group which admits a DF domain is simply one which admits Dirichlet domain with a line of centers and a cusp on the line of symmetry.

\textbf{(2)} The same discussion can also be used to show that these are the only Dirichlet centers giving rise to $P$. The Dirichlet center must be equidistant from a point of $\partial P$ and its destination under its side-pairing; in this set-up, the locus of such points is always precisely $L$. Thus, it is impossible to find a Fuchsian group with a triangle of Dirichlet centers all giving rise to the same domain.

\section{Reflection Groups}\label{refgps}

The goal of this section is to prove the main theorem. As a corollary, we will show that given the signature of any sphere which can be obtained as a quotient of $\mathbb{H}^2$, then we may exhibit a Fuchsian group $\Gamma$ which admits a Double Dirichlet domain (and a DF domain if there is at least one puncture) and gives rise to a quotient space of the given signature.\vspace{2 mm}\\
We first recall the following results regarding reflection groups (see \cite{Ratcliffe}, Section 7.1).

\begin{thm}\label{5.1} Let $G$ be a discrete reflection group with respect to the polygon $Q$. Then all the angles of $Q$ are submultiples of $\pi$, and if $g_S$ and $g_T$ are reflections in the adjacent sides $S$ and $T$ of $Q$ with $\theta(S,T) = \frac{\pi}{k}$, then $g_S \circ g_T$ has order $k$.\end{thm}

\begin{thm}\label{5.2} Let $Q$ be a finite-sided convex hyperbolic polygon of finite volume, all of whose angles are submultiples of $\pi$. Then the group $G$, generated by reflections of $\mathbb{H}^2$ in the sides of $Q$, is a discrete reflection group with respect to the polygon $Q$.\end{thm}

We will appeal to these results, as well as to the results of Sections $3$ and $4$, in the following discussion.

\begin{thm}\label{5.3} If the finitely generated, orientation-preserving, finite coarea Fuchsian group $\Gamma$ admits a Double Dirichlet domain, or a DF domain, $P$, then $\Gamma$ is an index $2$ subgroup of the discrete group $G$ of reflections in a hyperbolic polygon $Q$.\end{thm}

\begin{proof} Suppose first that $\Gamma$ admits a DF domain $P$. We know that $P$ has reflectional symmetry about a vertical axis $L$ which bisects $P$. Since $P$ is a fundamental domain for $\Gamma$, the side-pairings of $P$ generate $\Gamma$. Each side pairing, with the exception of the parabolic element pairing the vertical sides, has the form $\sigma_L \circ \sigma_i$, where $\sigma_L$ denotes reflection in $L$ and $\sigma_i$ is reflection in the isometric circle $S_i$, $1 \leq i \leq m$, where $S_i$ contains a side of $P$. Furthermore, since each side is paired with its mirror image in $L$, it suffices to consider the $\sigma_i$ corresponding to sides in one half of $P$. The parabolic side-pairing can be written $\sigma_L \circ \sigma_K$, where $\sigma_K$ is reflection in $K$, the vertical side of $P$ in the same half as the $S_i$. Thus we have a generating set for $\Gamma$ of the form
$$\{ \sigma_L \circ \sigma_1,\ \ldots \ , \sigma_L \circ \sigma_m, \sigma_L \circ \sigma_K \}$$
for some $m \in \mathbb{N}$. Consider the group $G$ obtained by adding the reflection $\sigma_L$ to this generating set. The set becomes
$$\{ \sigma_L, \sigma_L \circ \sigma_1,\ \ldots \ , \sigma_L \circ \sigma_m, \sigma_L \circ \sigma_K \}$$
and because $\sigma_L = \sigma_L^{-1}$ has order $2$, it follows that we can replace the generator $\sigma_L \circ \sigma_i$ with the element $\sigma_i$ and still have a generating set. The generating set
$$\{ \sigma_L, \sigma_1,\ \ldots \ , \sigma_m, \sigma_K \}$$
is precisely the set of reflections in the sides of a polygon $Q$ with sides on $K$, $L$ and $S_i$, $1 \leq i \leq m$. To prove that all of the angles of $Q$ are submultiples of $\pi$, it suffices to observe that the vertices of $P$ are paired symmetrically, and that the Poincar\'{e} Polyhedron Theorem gives that the sum of the angles in each cycle is $\frac{2\pi}{s}$, for $s \in \mathbb{N}$. Now Theorem \ref{5.2} allows us to reach the desired conclusion.\vspace{2 mm}\\
To prove the result for the case where $L \cap \textit{\r{P}}$ does not extend to $\partial \mathbb{H}^2$, we simply observe that in this case, every side-pairing generator of $\Gamma$ can be written $\sigma_L \circ \sigma_i$, since here there are no vertical sides. Instead of the cusp at $\infty$ we have another finite vertex of $P$, but since this vertex lies on the line $L$, it must also be an elliptic fixed point, and the paragraph above applies. \end{proof}

We now turn to the converse of the above result.

\begin{thm}\label{5.4} If $G$ is a discrete group of reflections in a polygon $Q \subset \mathbb{H}^2$, then $G$ contains an index $2$ subgroup of orientation-preserving isometries which admits a Double Dirichlet domain (and a DF domain if $Q$ has an ideal vertex at $\infty$).\end{thm}

\begin{proof} Let $Q$ be such a polygon. If necessary, rotate $Q$ so that one of its sides is vertical. Call this side $L$. By Theorem \ref{5.1}, all angles of $Q$ are submultiples of $\pi$. Denote by $\sigma_L$ reflection in the vertical side $L$ of $Q$. If $Q$ has another vertical side (and hence an ideal vertex at $\infty$), call this side $K$ and denote reflection in $K$ by $\sigma_K$. Denote by $\sigma_i$ reflection in the (non-vertical) line $S_i$ containing a side of $Q$. By definition, these reflections constitute a generating set for $G$. Let $\Gamma < G$ be the subgroup generated by elements of the form $\sigma_2 \circ \sigma_1$ where $\sigma_1$ and $\sigma_2$ are reflections in the generating set for $G$. Then $\Gamma$ is a Fuchsian group. Since $\sigma_L \notin \Gamma$, we see that the set $P := Q \cup \sigma_L Q$ is contained within a fundamental domain for $\Gamma$. We will show that $P$ is itself a fundamental domain for $\Gamma$.\vspace{2 mm}\\
To see this, denote by $T_i := \sigma_L(S_i)$ the geodesic obtained by reflecting $S_i$ in $L$. Then $T_i$ contains a side of $P$. Also denote $\sigma_L(K)$ by $M$. Then $K$ is paired with $M$ by the element $\sigma_L \circ \sigma_K \in \Gamma$, and $S_i$ is paired with $T_i$ by $\sigma_L \circ \sigma_i$. Thus the sides of $P$ are paired by generators of $\Gamma$. To see that these side-pairings generate $\Gamma$ themselves, consider a generating element $\sigma_2 \circ \sigma_1 \in \Gamma$. We may write 
$$\sigma_2 \circ \sigma_1 = \sigma_2 \circ (\sigma_L \circ \sigma_L) \circ \sigma_1 = (\sigma_2 \circ \sigma_L) \circ (\sigma_L \circ \sigma_1) = (\sigma_L \circ \sigma_2)^{-1} \circ (\sigma_L \circ \sigma_1),$$
which shows that together, the elements $\sigma_L \circ \sigma_i$ and $\sigma_L \circ \sigma_K$ generate $\Gamma$. We therefore have that $\Gamma$ has index $2$ in $G$, and that $P$ is a fundamental domain for $\Gamma$.\vspace{2 mm}\\
To see that $\Gamma$ admits a fundamental domain of the required type, it will suffice to check that $P$ is one. Let $z_0$ be any point on the line $L$ which lies in the interior of $P$. If there is a second vertical side $K$, then it is the line bisecting $z_0$ and $\sigma_K(z_0)$, so $\sigma_L(K) = M$ bisects $\sigma_L(z_0) = z_0$ and $\sigma_L(\sigma_K(z_0))$. Thus $M$ is a line of the form found in the definition on a Dirichlet domain centered at $z_0$. A similar argument applied to $(\sigma_L \circ \sigma_K)^{-1} = \sigma_K \circ \sigma_L$ shows that $K$ is also such a line. Now $S_i$ is the line bisecting $z_0$ and $\sigma_i(z_0)$, so $\sigma_L(S_i) = T_i$ bisects $\sigma_L(z_0) = z_0$ and $\sigma_L(\sigma_i(z_0))$. This shows that $T_i$ is a line of the form found in the definition on a Dirichlet domain centered at $z_0$. A similar argument shows that the same is true of $S_i$, and thus we see that $P$ must contain a Dirichlet domain centered at $z_0$. But we know that $P$ is itself a fundamental domain for $\Gamma$, so that $P$ is a Dirichlet domain for any center $z_0 \in L \cap \textit{\r{P}}$.\vspace{2 mm}\\
If there is a second vertical side $K$, we must also check that $P$ is a Ford domain. $S_i$ is the isometric circle of the generator $\sigma_L \circ \sigma_i$, and $T_i = \sigma_L(S_i)$ is the isometric circle of the inverse element. Since $\sigma_L \circ \sigma_K$ pairs the two vertical sides of $P$ and generates $\Gamma_\infty$, it follows that $P$ must contain a Ford domain for $\Gamma$. But $P$ is itself a fundamental domain, so this Ford domain cannot be a proper subset, and hence is equal to $P$.\end{proof}

We now show that Fuchsian groups with this symmetrical property, though they necessarily have genus zero, have no other restrictions on their signature.

\begin{cor}\label{cor} Given the signature $(0; n_1, \ldots , n_t ; m)$ of a (non-trivial, hyperbolic) sphere with $m \geq 0$ punctures and $t \geq 0$ cone points of orders $n_i \in \mathbb{N}$, for $1 \leq i \leq t$, there exists a Fuchsian group $\Gamma$ such that $\Gamma$ admits a Double Dirichlet domain (and a DF domain if $m > 0$) and $\mathbb{H}^2 / \Gamma$ is a sphere of the given signature.\end{cor}

\begin{proof} Suppose $m > 0$. Construct $Q$ by placing one vertex at $\infty$, $t$ vertices in $\mathbb{H}^2$ of angles $\frac{\pi}{n_i}$ ($n_i \geq 2$) for $1 \leq i \leq t$, and $m-1$ ideal vertices in $\mathbb{R}$. If $m=0$, construct a compact $t$-gon with angles $\frac{\pi}{n_1}, \ \ldots \ , \frac{\pi}{n_t}$. Now let $G$ be the group of hyperbolic isometries generated by reflections in the sides of $Q$. By Theorem \ref{5.4}, $\Gamma$ admits a DF domain $P = Q \cup \sigma Q$, where $\sigma$ denotes reflection in one of the vertical sides $L$ of $Q$. The symmetrical identifications, combined with the Poincar\'{e} Polyhedron Theorem, give that the quotient surface has the required signature.\end{proof}

\textbf{Remark.} If $m>0$ above, then there is a certain amount of freedom in our choice of the polygon $Q$. For example, we do not necessarily have to place one of the ideal vertices of $Q$ at $\infty$. We do so in order to ensure that we obtain a DF domain for $\Gamma$. Instead, we could have all of the ideal vertices lie in $\mathbb{R}$, thereby placing the line of symmetry $L$ away from any of the ideal vertices. Similarly, if $m > 1$, we could construct $Q$ so that $L$ meets only one of the $m$ ideal vertices, instead of $2$ in the construction above. We also do not have to construct $Q$ so that each angle is bisected by a vertical line; we only do so in order to demonstrate that it is possible to find the required polygon.

\section{A Non-congruence Maximal Arithmetic Reflection Group}\label{example}

In this section, we will prove explicitly that there exists a non-congruence maximal arithmetic hyperbolic reflection group. Recall that a non-cocompact hyperbolic reflection group $\Gamma_{\mathrm{ref}} < \mathrm{Isom}(\mathbb{H}^2)$ is called \emph{arithmetic} if and only if it is commensurable with $\mathrm{PSL}_2(\bZ)$. Such a group is then called \emph{congruence} if it contains some principal congruence subgroup
\[ \Gamma(n) = \left\lbrace \begin{pmatrix}a & b \\ c & d\end{pmatrix} \ \Big| \ a \equiv d \equiv \pm 1, b \equiv c \equiv 0 \ \mathrm{mod} \ n \right\rbrace \subset \mathrm{PSL}_2(\bZ). \]

Consider the group $\Gamma < \mathrm{PSL}_2(\mathbb{R})$ generated by the matrices
\[ \gamma_1=\begin{pmatrix}1 & 1 \\ 0 & 1 \end{pmatrix}, \gamma_2 = \begin{pmatrix}0 & \frac{-1}{\sqrt{11}} \\ \sqrt{11} & 0 \end{pmatrix}, \gamma_3 =  \begin{pmatrix}\sqrt{11} & \frac{5}{\sqrt{11}} \\ 2\sqrt{11} & \sqrt{11} \end{pmatrix}, \gamma_4 = \begin{pmatrix}10 & 3 \\ 33 & 10 \end{pmatrix}, \gamma_5 = \begin{pmatrix}23 & 8 \\ 66 & 23 \end{pmatrix}. \]
We first wish to show that $\Gamma$ is discrete. Consider the group
\[ \Gamma_0(11) = \left\lbrace \begin{pmatrix}a & b \\ 11c & d\end{pmatrix} \ \Big| \ a, b, c, d \in \bZ, \ ad-11bc=1 \right\rbrace \subset \mathrm{PSL}_2(\bZ).\]
It is well-known (\cite{ternary}, \cite{LMR}) that the normalizer $N(\Gamma_0(11))$ of $\Gamma_0(11)$ in $\mathrm{PSL}_2(\mathbb{R})$ is a (maximal arithmetic) Fuchsian group generated by $\Gamma_0(11)$ and 
\[ \begin{pmatrix}0 & -\frac{1}{\sqrt{11}} \\ \sqrt{11} & 0 \end{pmatrix}, \]
which is $\gamma_2 \in \Gamma$. We see then that 
\[ \begin{pmatrix}0 & -\frac{1}{\sqrt{11}} \\ \sqrt{11} & 0 \end{pmatrix} \begin{pmatrix} 2 & 1 \\ -11 & -5 \end{pmatrix} = \begin{pmatrix}\sqrt{11} & \frac{5}{\sqrt{11}} \\ 2\sqrt{11} & \sqrt{11} \end{pmatrix} = \gamma_3 \in \Gamma, \]
and since $\gamma_1, \gamma_4, \gamma_5 \in \Gamma_0(11)$, we have that $\Gamma < N(\Gamma_0(11))$, and so $\Gamma$ is discrete. We next wish to construct a Ford domain for $\Gamma$. In Figure \ref{ford01} we see the isometric circles corresponding to the generators listed above and their inverses.
\begin{figure}[htb]
\begin{center}
\includegraphics[scale=0.75]{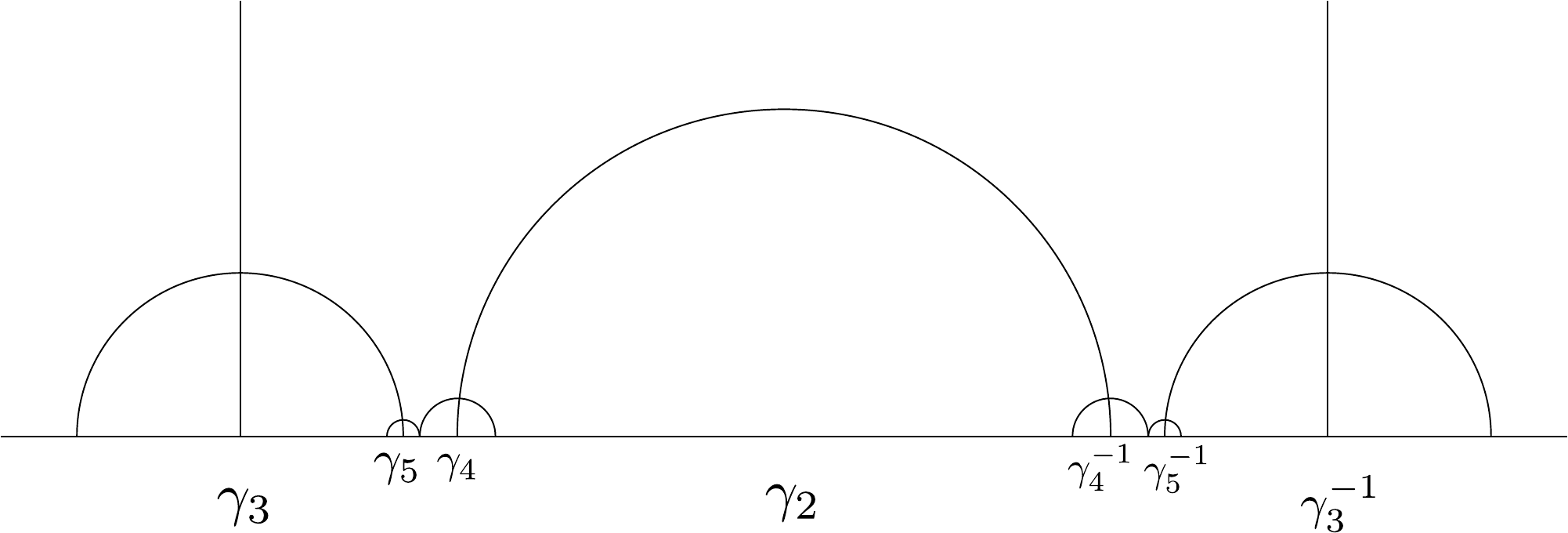}
\caption{\label{ford01}A Ford domain for $\Gamma$}
\end{center}
\end{figure}
The claim is that this polygon is in fact the required Ford domain. To see this, observe that each generator $\gamma_i$ can be decomposed into the product of two reflections $\gamma_i = \sigma_L \circ \sigma_i$, where $\sigma_i$ is reflection in the isometric circle $S_i$ of $\gamma_i$, $\sigma_1$ is reflection in the line $x=-\frac{1}{2}$, and $\sigma_L$ is reflection in the line $x=0$. Thus the elements of the generating set for $\Gamma$ pair the sides of $P$, and each pushes $\textit{\r{P}}$ completely off itself. This shows that $P$ is a fundamental domain for $\Gamma$; by its construction, it is a Ford domain.

Thus we see that the quotient space $\mathbb{H}^2 / \Gamma$ is a sphere of signature $(0; 2, 2, 2, 2; 2)$ and area $4\pi$. Further, $P$ is a DF domain, as each of these generators pairs one side $S_i$ of $P$ with its reflection $\sigma_L(S_i)$ in the line $x=0$. Thus, by Theorem \ref{5.3}, we see that $\Gamma$ is the index $2$ orientation-preserving subgroup of the group $\Gamma_{\mathrm{ref}}$ of reflections in a hyperbolic hexagon $Q$ with angles $(0, \frac{\pi}{2}, \frac{\pi}{2}, 0, \frac{\pi}{2}, \frac{\pi}{2})$. The claim is that this hyperbolic reflection group $\Gamma_{\mathrm{ref}}$ is arithmetic, maximal (as an arithmetic reflection group), and non-congruence.\vspace{2mm}

\textbf{Claim 1.} \textit{$\Gamma_{\mathrm{ref}}$ is arithmetic.}

\begin{proof}Since $\Gamma_{\mathrm{ref}}$ is not cocompact, we need to show that it is commensurable with $\mathrm{PSL}_2(\bZ)$. The group $G = \Gamma \cap \mathrm{PSL}_2(\bZ)$ is not equal to $\Gamma$, by the presence of the non-integral elements $\gamma_2$ and $\gamma_3$. However, it contains the matrices
\[ \gamma_1 = \begin{pmatrix}1 & 1 \\ 0 & 1 \end{pmatrix}, \ \ \gamma_2 \gamma_1^{-1} \gamma_2 = \begin{pmatrix} 1 & 0 \\ 11 & 1 \end{pmatrix}, \ \ \gamma_2 \gamma_3 = \begin{pmatrix}-2 & -1 \\ 11 & 5 \end{pmatrix},\]
\[ \gamma_2 \gamma_3^{-1} = \begin{pmatrix}2 & -1 \\ 11 & -5 \end{pmatrix}, \ \ \gamma_4 = \begin{pmatrix}10 & 3 \\ 33 & 10 \end{pmatrix}, \ \ \gamma_5 = \begin{pmatrix}23 & 8 \\ 66 & 23 \end{pmatrix}. \]
The isometric circles of these elements and their inverses are shown in Figure \ref{ford02}. Notice that the isometric circles centered at $\frac{2}{11}$ and $-\frac{2}{11}$ are paired with those at $\frac{5}{11}$ and $-\frac{5}{11}$ respectively; with these four circles excepted, each other side is paired with its reflection in the line $x=0$. There are four equivalence classes of ideal points: these classes are $\lbrace \infty \rbrace$, $\lbrace 0 \rbrace$, $\lbrace \frac{1}{3}, -\frac{1}{3} \rbrace$, $\lbrace \frac{4}{11}, \frac{3}{11}, -\frac{3}{11}, -\frac{4}{11} \rbrace$. All four finite vertices belong to the same cycle, and their angles are $\pi/3$ at $x=\pm\frac{1}{2}$, and $2\pi/3$ at $x=\pm\frac{3}{22}$, giving angle sum $2\pi$.
\begin{figure}[htb]
\begin{center}
\includegraphics[scale=0.75]{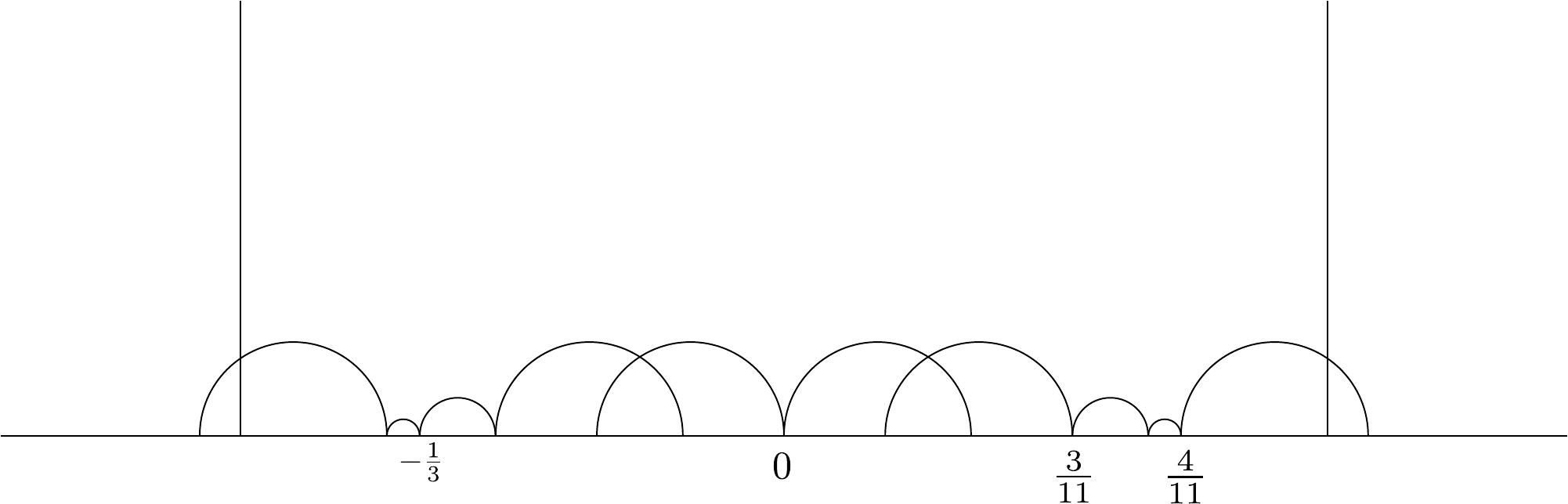}
\caption{\label{ford02}A Ford domain for $G$}
\end{center}
\end{figure}
The region $P_G$ of $\mathbb{H}^2$ bounded by these circles and the lines $x=-\frac{1}{2}$ and $x=\frac{1}{2}$ has area $8\pi$ and contains a Ford domain for $G$. This is enough for us to conclude that it is a Ford domain for $G$: since $G$ is a proper subgroup of $\Gamma$, of finite index due to the finite area of $P_G$, $G$ must have coarea a multiple $4m\pi$ of $4\pi$, where $m=\lbrack \Gamma : G \rbrack > 1$. That the area of $P_G$ is $8\pi$ tells us that $m \leq 2$, and hence that in fact $m=2$. So $G$ has index $2$ in $\Gamma$. Thus, $\Gamma_{\mathrm{ref}}$ shares the finite index subgroup $G$ with $\mathrm{PSL}_2(\bZ)$.\end{proof}
\newpage

\textbf{Claim 2.} \textit{$\Gamma_{\mathrm{ref}}$ is a maximal reflection group.}

\begin{proof}If $\Gamma_{\mathrm{ref}}$ were not maximal, it would be properly contained in another reflection group $H_{\mathrm{ref}}$, which is therefore also arithmetic. Let $H < H_{\mathrm{ref}}$ denote the orientation-preserving index $2$ subgroup. Note that then we have $\Gamma < H$. Since $\Gamma$ and $H$ are both arithmetic Fuchsian groups of genus zero, they are contained in a common maximal, arithmetic, genus zero Fuchsian group $M$ from the appropriate list in \cite{LMR}. As we saw above, $\Gamma$ is contained in the normalizer $N(\Gamma_0(11))$, and by area considerations we find that $\lbrack N(\Gamma_0(11)):\Gamma \rbrack=2$. Further, $\Gamma$ cannot be contained in any other of these maximal arithmetic groups; to see this, observe that if $n \neq 11$ then, if we pick some non-zero integer $b$ coprime to $n$, we may find integers $a, d$ such that $\begin{pmatrix}a & b \\ n & d\end{pmatrix} \in \Gamma_0(n)$. We then have
\begin{align*} 
\gamma_2 \begin{pmatrix}a & b \\ n & d\end{pmatrix} \gamma_2 &= \begin{pmatrix}0 & -\frac{1}{\sqrt{11}} \\ \sqrt{11} & 0 \end{pmatrix} \begin{pmatrix}a & b \\ n & d\end{pmatrix} \begin{pmatrix}0 & -\frac{1}{\sqrt{11}} \\ \sqrt{11} & 0 \end{pmatrix}\\
 &= \begin{pmatrix} -d & \frac{n}{11} \\ 11b & -a\end{pmatrix}.
 \end{align*}
We wish to show that this does not belong to $\Gamma_0(n)$. If $n$ is not divisible by $11$ this is clear, so suppose $n \ge 22$ is a multiple of $11$. Then, by construction, $b$ is coprime to $11$, and so $11b$ is not divisible by $n$. This shows that $\gamma_2$ cannot belong to any normalizer $N(\Gamma_0(n))$ except $N(\Gamma_0(11))$.

It remains to verify that we cannot have $H = M = N(\Gamma_0(11))$. Construction of the Ford domain for $N(\Gamma_0(11))$ yields the generating set
\[ \begin{pmatrix}1 & 1 \\ 0 & 1 \end{pmatrix}, \begin{pmatrix}0 & \frac{-1}{\sqrt{11}} \\ \sqrt{11} & 0 \end{pmatrix}, \begin{pmatrix}\sqrt{11} & \frac{5}{\sqrt{11}} \\ 2\sqrt{11} & \sqrt{11} \end{pmatrix}, \begin{pmatrix}\sqrt{11} & \frac{-4}{\sqrt{11}} \\ 3\sqrt{11} & -\sqrt{11} \end{pmatrix}, \begin{pmatrix}-\sqrt{11} & \frac{-4}{\sqrt{11}} \\ 3\sqrt{11} & \sqrt{11} \end{pmatrix}; \]
the Ford domain corresponding to these generators is shown in Figure \ref{ford03}.
\begin{figure}[htb]
\begin{center}
\includegraphics[scale=0.65]{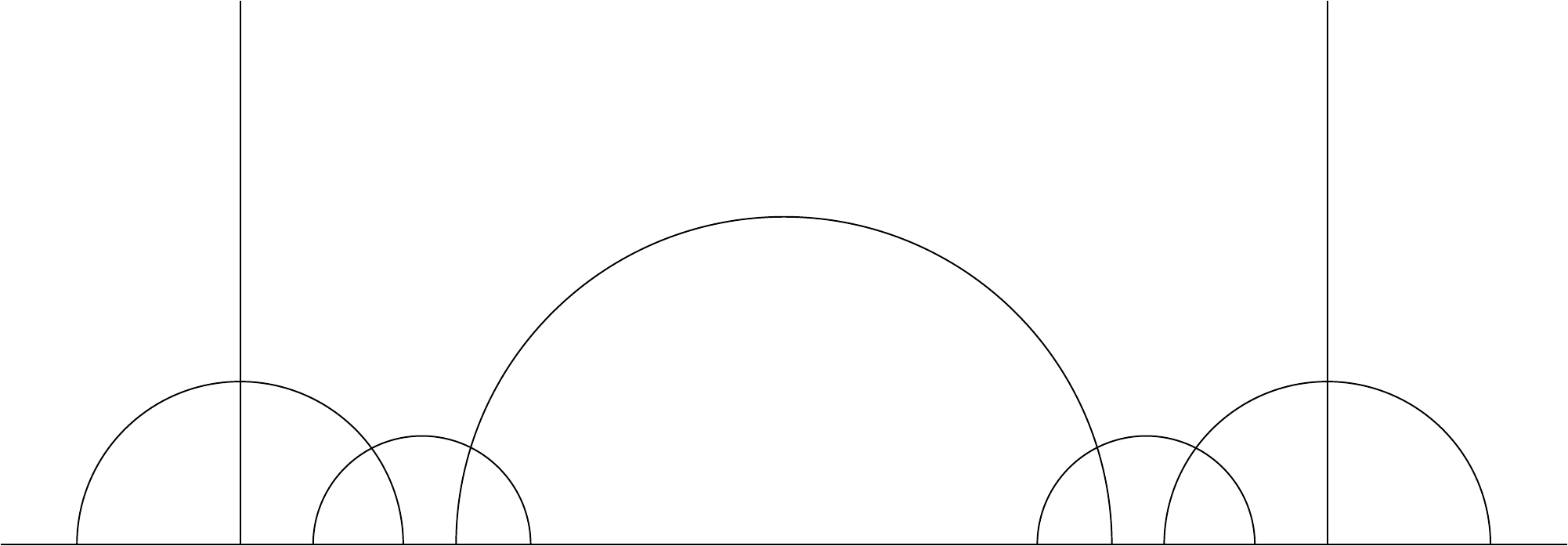}
\caption{\label{ford03}A Ford domain for $N(\Gamma_0(11))$}
\end{center}
\end{figure}
The fact that three of the generating elements are involutions, which pair adjacent sides of the Ford domain, precludes $N(\Gamma_0(11))$ from possessing a DF domain. By Theorem \ref{5.4}, this also precludes it from being an index $2$ subgroup of a reflection group. Thus $\Gamma_{\mathrm{ref}}$ is maximal.\end{proof}

\textbf{Remark.} Since $\Gamma_{\mathrm{ref}}$ is a maximal arithmetic hyperbolic reflection group, one would expect to find it in existing lists of such groups. This example appears to be the lattice $2$-fill$(L_{26.1})$ in Allcock's enumeration \cite{Allcock} of rank $3$ reflective Lorentzian lattices, which would correspond to the case $N=26$ in Nikulin's table $1$ \cite{Nikulin}. If one could show $\Gamma_{\mathrm{ref}}$ is indeed this lattice, this would provide an alternative proof that it is maximal arithmetic; however, we omit this at present, as the proofs given above suffice for our purposes. \vspace{2mm}

\textbf{Claim 3.} \textit{$\Gamma_{\mathrm{ref}}$ is not congruence.}

\begin{proof}Suppose $\Gamma_{\mathrm{ref}}$ is congruence. Then it contains some principal congruence subgroup $\Gamma(n)$. These groups all belong to the modular group, so $\Gamma(n)$ is contained in $G = \Gamma \cap \mathrm{PSL}_2(\bZ)$. By Wohlfahrt's Theorem (see \cite{IntMat}, p. $149$), $G$ contains $\Gamma(n)$ for $n$ equal to the level of $G$; i.e. the smallest natural number such that $G$ contains the normal closure of
\[ T^n = \begin{pmatrix} 1 & 1 \\ 0 & 1 \end{pmatrix}^n = \ \begin{pmatrix} 1 & n \\ 0 & 1 \end{pmatrix} \]
in $\mathrm{PSL}_2(\bZ)$.\\

\textbf{Subclaim.} \textit{The level of $G$ is $11$.}

\begin{proof}To prove this, we need to show that given any $\varphi \in \mathrm{PSL}_2(\bZ)$, we have that $\varphi \, T^{11} \varphi^{-1} \in G$. If $\varphi$ fixes $\infty$ this is clear, so suppose $\varphi(\infty) \neq \infty$. Topologically, $\mathbb{H}^2 / G$ is a torus with four cusps, with the cusp orbits in $\mathbb{Q} \cup \lbrace \infty \rbrace$ represented by $0$, $\infty$, $\frac{1}{3}$ and $\frac{3}{11}$. Therefore $\varphi(\infty)$ is $G$-equivalent to exactly one of these four points; let $g \in G$ be such that $g^{-1} \varphi(\infty)$ is this point. We observe that $T^{11} \in G$; we also find that
\[ \begin{pmatrix} 0 & -1 \\ 1 & 0\end{pmatrix} \begin{pmatrix} 1 & 11 \\ 0 & 1\end{pmatrix} \begin{pmatrix} 0 & -1 \\ 1 & 0\end{pmatrix} = \begin{pmatrix} 1 & 0 \\ -11 & 1\end{pmatrix} = \begin{pmatrix} 1 & 0 \\ 11 & 1\end{pmatrix}^{-1} \in G  \]
is a parabolic element fixing $0$,
\begin{align*} \begin{pmatrix} 1 & 0 \\ 3 & 1\end{pmatrix} \begin{pmatrix} 1 & 11 \\ 0 & 1\end{pmatrix} \begin{pmatrix} -1 & 0 \\ 3 & -1\end{pmatrix} &= \begin{pmatrix} -32 & 11 \\ -99 & 34\end{pmatrix} \\
&= \begin{pmatrix} 10 & 3 \\ 33 & 10\end{pmatrix} \begin{pmatrix} -23 & 8 \\ 66 & -23\end{pmatrix} \in G
\end{align*}
is a parabolic element fixing $\frac{1}{3}$, and
\begin{align*} \begin{pmatrix} 3 & 11 \\ 1 & 4\end{pmatrix} \begin{pmatrix} 1 & 1 \\ 0 & 1\end{pmatrix} \begin{pmatrix} -4 & 1 \\ 11 & -3\end{pmatrix} &= \begin{pmatrix} -32 & 9 \\ -121 & 34\end{pmatrix} \\
&= \begin{pmatrix}2 & -1 \\ 11 & -5\end{pmatrix} \begin{pmatrix} 23 & 8 \\ 66 & 23\end{pmatrix} \begin{pmatrix}-5 & -1 \\ 11 & -2\end{pmatrix}\begin{pmatrix} -10 & 3 \\ 33 & -10\end{pmatrix} \in G
\end{align*}
is a parabolic element fixing $\frac{3}{11}$. Note that in this last case, $G$ also contains a conjugate of $T^{11}$ fixing $\frac{3}{11}$, by taking the $11$th power of the given element. Thus there exists a conjugate $\alpha$ of $T^{11}$ such that $\alpha \in G$ and $\alpha$ fixes $g^{-1} \varphi(\infty)$. The element $g.\alpha.g^{-1} \in G$ is therefore a parabolic element, conjugate in $\mathrm{PSL}_2(\bZ)$ to $T^{11}$, with parabolic fixed point at $\varphi(\infty)$. We wish to show that $g . \alpha . g^{-1} = \varphi \, T^{11} \varphi^{-1}$. Since the former element is known to be a conjugate of $T^{11}$, we may alternatively write it as $\psi \, T^{11} \psi^{-1}$ for some $\psi \in \mathrm{PSL}_2(\bZ)$ with $\psi(\infty) = \varphi(\infty)$. Now $\psi^{-1} \varphi \in \mathrm{PSL}_2(\bZ)$ fixes $\infty$ and so must be a power of $T$; in particular, $\psi^{-1} \varphi$ commutes with $T$. It follows that
\[ \psi^{-1} \varphi \, T^{11} \varphi^{-1} \psi = T^{11} \]
and therefore
\[ g.\alpha.g^{-1} = \psi \, T^{11} \psi^{-1} = \varphi \, T^{11} \varphi^{-1} \]
as required. Thus $G$ contains all elements of the form $\varphi \, T^{11} \varphi^{-1}$, and so the level of $G$ is at most $11$. To see that it is not smaller, observe that $G$ does not contain any element of the form
\[ \begin{pmatrix}1 & 0 \\ t & 1\end{pmatrix}\]
for $t = 1, 2, \ldots, 10$. \end{proof}

To complete the proof of Claim $3$, we note that by the Subclaim, $G$ must contain $\Gamma(11)$. Computation in Gap \cite{Gap4} reveals that the core of $G$ in $\mathrm{PSL}_2(\bZ)$, the largest normal subgroup of $\mathrm{PSL}_2(\bZ)$ contained in $G$, has index $k= 1351680=2^{13} \cdot 3 \cdot 5 \cdot 11$ in $\mathrm{PSL}_2(\bZ)$. Thus $G$ cannot contain a normal subgroup of $\mathrm{PSL}_2(\bZ)$ of index (in $\mathrm{PSL}_2(\bZ)$) smaller than this constant. But all principal congruence subgroups are normal, and $\lbrack \mathrm{PSL}_2(\bZ) : \Gamma(11) \rbrack = 660 < k$. From this contradiction we conclude that $\Gamma_{\mathrm{ref}}$ is not congruence.\end{proof}

\textbf{Remark.} Hsu \cite{Hsu} gives a congruence test which can be applied to $G$. Since $G$ has index $24$ in $\mathrm{PSL}_2(\bZ)$, we obtain a representation in the symmetric group $S_{24}$. After expressing the known generators for $G$ in terms of $L= \begin{pmatrix}1 & 1\\0 & 1\end{pmatrix}$ and $R=\begin{pmatrix}1 & 0 \\ 1 & 1\end{pmatrix}$, we find
\[ L=(2 \ 4 \ 9 \ 15 \ 8 \ 5 \ 11 \ 13 \ 7 \ 3 \ 6)(10 \ 17 \ 21 \ 23 \ 22 \ 19 \ 14 \ 12 \ 18 \ 20 \ 16) \]
and
\[ R = (1 \ 2 \ 5 \ 12 \ 14 \ 7 \ 4 \ 10 \ 16 \ 8 \ 3)(9 \ 17 \ 19 \ 13 \ 11 \ 18 \ 21 \ 24 \ 22 \ 20 \ 15) \]
are both of order $11$, also giving that the level of $G$ is $11$. Hsu's test is then that $G$ is congruence if and only if $(R^2 L^{-\frac{1}{2}})^3 = 1$, where $\frac{1}{2}$ is the multiplicative inverse of $2$ mod $11$, in this case equal to $6$. We find that $R^2 L^{-6}$ has order $6$, and so $G$ is non-congruence.\vspace{2mm}

\section{Kleinian Groups and DF Domains}\label{kleinian}

In this section, it will be shown that only one direction (the analogue of \ref{5.4}) of the main theorem holds when we consider Kleinian groups in the place of Fuchsian groups. This is because the added dimension gives new possibilities for the shape of the domains in question; in particular, they no longer have to glue up in a completely symmetrical way, although some symmetry remains. Examples will be given to demonstrate this flexibility, which extends as far as having non-trivial cuspidal cohomology. The discussion will be restricted to DF domains; as the above work demonstrates, it is not unreasonable to suppose that Double Dirichlet domains share many similar properties.

\begin{thm}\label{6.1}Let $Q \subset \mathbb{H}^3$ be a finite-sided, convex hyperbolic polyhedron satisfying the hypotheses of Theorem \ref{5.2}, and let $G$ be the discrete group of reflections in $Q$. Then G contains an index $2$ Kleinian subgroup which admits a Double Dirichlet domain (and a DF domain if Q has an ideal vertex).
\end{thm}

\begin{proof}Suppose that $Q$ is placed in upper half-space $\mathbb{H}^3$ such that one of its faces $L$ is contained in a vertical plane. Let
$$G = < \tau_1, \ldots, \tau_m, \tau_L >$$
be a generating set for $G$. Let
$$\Gamma = < \tau_L \circ \tau_1, \ldots, \tau_L \circ \tau_m >$$
be the index $2$ subgroup. Let $P = Q \cup \tau_L Q$. Let $w_0 = x_0 + y_0 i + z_0 j \in \textit{\r{L}}$, for $z_0 > 0$. The claim is that $w_0$ is a Dirichlet center for $\Gamma$. Fix a generator $\gamma_i = \tau_L \circ \tau_i$. Then the plane $P_i$ fixed by $\tau_i$ bisects $w_0$ and $\tau_i(w_0)$, and so $\tau_L(P_i)$, which by construction is a face of $P$, bisects $w_0$ and $\gamma_i(w_0)$.
\end{proof}

The next result provides the first counterexamples of Theorem \ref{5.3} by exhibiting Kleinian groups which admit DF domains and do not have index $2$ in a reflection group.

\begin{prop}Let $Q$ be an all-right hyperbolic polyhedron, with a vertex at $\infty$, and all vertices ideal. Let $G$ be the group of reflections in $Q$. Then $G$ contains a subgroup of index $4$ which admits a DF domain.
\end{prop}

\begin{proof}Since $Q$ is all-right, the link of each vertex is a rectangle. Rotate $Q$ in $\mathbb{H}^3$ so that the four vertical sides, which meet at the vertex at $\infty$, each lie above vertical or horizontal lines in $\mathbb{C}$. Let $H$ be a vertical side, $V$ a horizontal side, and $\tau_H$ and $\tau_V$ the respective reflections. Let $P = (Q \cup \tau_H Q) \cup \tau_V(Q \cup \tau_H Q)$. Then $P$ is the union of $4$ copies of $Q$. Looking down from $\infty$ on the floor of $P$, label by $A$ the non-vertical face adjacent to the top-left vertex and to the vertical face opposite $H$. Label any non-vertical faces adjacent to this face $B$. Proceed to label every non-vertical face $A$ or $B$, with no two adjacent faces sharing the same label. The symmetry of $P$ implies that this labeling is symmetric in both horizontal and vertical directions. Define the subgroup $\Gamma$ as follows. Given a non-vertical side $P_i$ of $P$, if $P_i$ has label $A$, let the element $\tau_H \circ \tau_i$ belong to $\Gamma$; if $P_i$ has label $B$, let $\tau_V \circ \tau_i$ belong to $\Gamma$. If $H'$ is the face opposite $H$, and $V'$ opposite $V$, let $\tau_H \circ \tau_{H'}$ and $\tau_V \circ \tau_{V'}$ belong to $\Gamma$. Then $P$ is a DF domain for $\Gamma$.
\end{proof}

{\bf Remark.} Given a group $\Gamma$ constructed as in the above proof, note that $\Gamma$ is not an index $2$ subgroup of the group of reflections in the polyhedron $(Q \cup \tau_H Q)$. This is because the reflection $\tau_H$ will be absent from this group, preventing the construction of elements of $\Gamma$ of the form $\tau_H \circ \tau_i$. The same is valid for the group of reflections in the polyhedron $(Q \cup \tau_V Q)$. \vspace{2 mm}\\
Since there is no direct analogue of Theorem \ref{5.3} for Kleinian groups, the question arises as to what, if anything, is implied about a Kleinian group by it having a DF domain. For example, one might ask whether such groups must have trivial cuspidal cohomology. The following example gives a Kleinian group which admits a DF domain, but which has non-trivial cuspidal cohomology; that is, there exists a non-peripheral homology class of infinite order in the first homology of the quotient space.\vspace{2 mm}\\
{\bf Example.} Let $\Gamma < \mathrm{PSL}_2(\mathbb{C})$ be generated by the matrices
$$\begin{pmatrix}1 & 5 \\ 0 & 1\end{pmatrix}, \begin{pmatrix}1 & 5i \\ 0 & 1\end{pmatrix}, \begin{pmatrix}0 & \frac{-1}{\sqrt{2}} \\ \sqrt{2} & 0\end{pmatrix}, \begin{pmatrix}-\sqrt{2} & \frac{i}{\sqrt{2}} \\ -i\sqrt{2} & -\sqrt{2}\end{pmatrix},$$
and
$$\begin{pmatrix}1 & a \\ 0 & 1\end{pmatrix}\begin{pmatrix}0 & \frac{-1}{\sqrt{2}} \\ \sqrt{2} & 0\end{pmatrix}\begin{pmatrix}1 & \bar{a} \\ 0 & 1\end{pmatrix} = \begin{pmatrix}\sqrt{2}a & \sqrt{2}a\bar{a}-\frac{1}{\sqrt{2}} \\ \sqrt{2} & \sqrt{2}\bar{a}\end{pmatrix}$$
for each $a \in \lbrace 1, 2, 1+i, 2+i, 2i, 1+2i, 2+2i, 1-i, 2-i, -2i, 1-2i, 2-2i \rbrace$, where $\bar{a}$ is the complex conjugate of $a$. Then the isometric spheres of these matrices have centers at the Gaussian integers $\lbrace x+iy \ | \ x, y \in \mathbb{Z} \rbrace$ and radius $\frac{1}{\sqrt{2}}$. The square with vertices at $\pm \frac{5}{2} \pm \frac{5}{2}i$ is a Dirichlet domain for the action of $\Gamma_\infty$. Let $P$ be the intersection of the exterior of all these isometric spheres with the chimney above the given rectangle. Then $P$ is a DF domain for $\Gamma$, with Dirichlet center any point of $\textit{\r{P}}$ above $0$. Every dihedral angle of $P$ is $\frac{\pi}{2}$. The quotient space $\mathbb{H}^3 / \Gamma$ has $14$ boundary components; the cusp at $\infty$ gives a boundary torus, and each of the $13$ cusp cycles in $\mathbb{C}$ gives a $(2,2,2,2)$ or a $(2,4,4)$ sphere. Thus the peripheral homology has rank $1$. Computation using Gap \cite{Gap4} gives that $H_1(\mathbb{H}^3 / \Gamma)$ has rank $2$, so there is infinite non-peripheral homology.\vspace{2 mm}\\
{\bf Remarks. (1)} The cuspidal cohomology of this example has rank $1$, but it can be modified to give examples where this rank is arbitrarily high. 

{\bf (2)} This example is arithmetic. To see this, consider the Picard group $\mathrm{PSL}_2(\mathcal{O}_1)$. This group contains as a finite index subgroup the congruence subgroup $\Gamma_0(2)$, where the lower left entry is a member of the ideal in $\mathcal{O}_1$ generated by $2$. This congruence subgroup is normalized by the element
\[\gamma = \begin{pmatrix}0 & \frac{-1}{\sqrt{2}} \\ \sqrt{2} & 0\end{pmatrix}\]
and, since $\gamma$ has order $2$, $\Gamma_0(2)$ is an index $2$ subgroup of the group $H$ obtained by adding $\gamma$. The group $\Gamma$ given in the example is a subgroup of $H$, of finite index due to the finite volume of the DF domain. In turn, $H$ is commensurable with $\mathrm{PSL}_2(\mathcal{O}_1)$, as both share $\Gamma_0(2)$ as a subgroup of finite index.

{\bf (3)} The quotient space of $\mathbb{H}^3$ by this group is not a manifold, so one can thus ask whether there exists another example which has non-trivial cuspidal cohomology, and which is additionally torsion-free.\vspace{2 mm}\\
Although there does not appear to be a specific condition for a Kleinian group which is equivalent to having a DF domain, we can say something about a group which admits a DF domain. We cannot always decompose an orientation-preserving isometry of $\mathbb{H}^3$ into the composition of two reflections, but Carath\'{e}odory \cite{Car} shows that we need at most four. If $\gamma \notin \Gamma_\infty$, these can be taken to be $\gamma = \gamma_4 \circ \gamma_3 \circ \gamma_2 \circ \gamma_1$, where $\gamma_1$ is reflection in the isometric sphere $S_\gamma$, $\gamma_2$ in the vertical plane $R_\gamma$ bisecting $S_\gamma$ and $S_{\gamma^{-1}}$, and $\gamma_4 \circ \gamma_3$ is rotation around the vertical axis through the North pole of $S_{\gamma^{-1}}$.

\begin{thm}\label{dim3thm}Suppose the Kleinian group $\Gamma$ admits a DF domain $P$. Then the planes $R_\gamma$, for side-pairings $\gamma \in \Gamma \setminus \Gamma_\infty$ of $P$, all intersect in a vertical axis. Furthermore, for each such $\gamma$, $\gamma_4 \circ \gamma_3 = 1$,  and so each element of the corresponding generating set for $\Gamma$ has real trace.
\end{thm}

\begin{proof}Let $P$ be a Ford domain. Suppose there is some side-pairing $\gamma$ such that $\gamma_4 \circ \gamma_3 \neq 1$. By considering the North pole of $S_\gamma$ and its image, the North pole of $S_{\gamma^{-1}}$, we see that if $P$ were a Dirichlet domain, its center $w_0$ would have to be in the plane $R_\gamma$. But given any such choice of $w_0$, one can find a point $w \in P \cap S_\gamma$ such that $d(w_0,w) \neq d(w_0,\gamma(w))$. Thus $P$ is not a Dirichlet domain. Since each $\gamma \in \Gamma \setminus \Gamma_\infty$ is then simply the composition of two reflections, it is the conjugate in $\mathrm{PSL}_2(\mathbb{C})$ of an element of $\mathrm{PSL}_2(\mathbb{R})$. It thus has real trace. Since it is assumed that any element of $\Gamma_\infty$ is parabolic, these too have real trace.\vspace{2 mm}\\
Next suppose that the planes $R_\gamma$ do not have a common intersection. Since we know that $\gamma_4 \circ \gamma_3 = 1$, for a given $\gamma$, the plane $R_\gamma$ represents the set of potential Dirichlet centers. If there is no common such center, $P$ is not a Dirichlet domain. Thus if $P$ is a DF domain, the planes $R_\gamma$ have a common intersection.
\end{proof}

The examples given earlier in this section give a flavor of the particular case with only two distinct, perpendicular planes $R_\gamma$. It is therefore possible for DF domains to be more complicated than this. This theorem provides a useful criterion for having a DF domain, which can be used to check known Ford domains. Observe that the vertical axis of intersection of the planes $R_\gamma$ must correspond to a Dirichlet center for the action of $\Gamma_\infty$. Thus we see that the figure-$8$ knot group \cite{Riley}, as well as the Whitehead link group and the group of the Borromean rings \cite{Wiel1} do not admit DF domains. Furthermore, the groups obtained from a standard Ford domain in \cite{Wiel2} cannot admit DF domains. Although in some cases, with the right choice of Ford domain, one can generate congruence subgroups of Bianchi groups using elements of real trace, the sides of the domain are identified in a way similar to the corresponding Fuchsian congruence subgroup, and so these groups seldom admit a DF domain.

\bibliographystyle{amsplain}
\bibliography{dfdomainrefs}
Department of Mathematics,\\
University of Texas\\
Austin, TX 78712, USA.\\
Email: glakeland@math.utexas.edu

\end{document}